\documentclass[11pt, oneside]{article}   	
\usepackage{geometry}                		
\usepackage{graphicx}				
\usepackage{algorithm2e}
\RestyleAlgo{ruled}
\usepackage{xcolor}					
\usepackage{comment}
\usepackage{tikz}
\usepackage{hyperref}
\usepackage{cite}
		
\usepackage{amssymb}
\usepackage{amsmath}
\usepackage{amsthm}

\newtheorem{theorem}{Theorem}[section]
\newtheorem{lemma}[theorem]{Lemma}

\newtheorem{corollary}[theorem]{Corollary}

\newtheorem{remark}[theorem]{Remark}

\theoremstyle{definition}

\title{When is a planar rod configuration infinitesimally rigid?}
\author{Signe Lundqvist, Klara Stokes and Lars-Daniel \"Ohman\\Ume{\aa} University, Sweden}
\date{}							

\begin{document}
\maketitle

\begin{center}
\textbf{Abstract}
\end{center}

In this article, we provide a way of determining the infinitesimal rigidity of rod configurations realizing a rank two incidence geometry in the Euclidean plane. We model each rod with a cone over its point set, and show that infinitesimal rigidity of generic frameworks of the resulting graph is equivalent to infinitesimal rigidity of sufficiently generic rod configurations realizing the incidence geometry. The generic rigidity of graphs in the plane is characterized by a counting condition, so this gives a combinatorial characterization of the infinitesimal rigidity of rod configurations.

Rank two incidence geometries are also known as hypergraphs. In 1984, Tay and Whiteley conjectured that the infinitesimal rigidity of rod configurations realizing $2$-regular hypergraphs is determined by the rigidity of generic body-and-hinge frameworks realizing the same hypergraph. Their conjecture is commonly known as the molecular conjecture because of its application to molecular chemistry. Jackson and Jordán proved the molecular conjecture in the plane in 2006 and in 2011 Katoh and Tanigawa proved the molecular conjecture in arbitrary dimension. 

In 1989, Whiteley proved a version of the molecular conjecture for hypergraphs of arbitrary degree that have realizations as independent body-and-joint frameworks, using the 2-plane matroid. Our result is an extension of Whiteley's result to arbitrary hypergraphs.

\section{Introduction}

The mathematical theory of structural rigidity has a long history. For example, in the nineteenth century, Cauchy studied rigidity of polyhedra, and Maxwell studied graph frameworks \cite{Cauchy, Maxwell}. Rigidity of graph frameworks has since been well-studied, in various contexts. Combinatorial rigidity theory is also concerned with geometric realizations of other combinatorial structures, for example geometric realizations of hypergraphs. A comprehensive review of results and topics in combinatorial rigidity theory, as well as relevant definitions, can be found in a review article by Nixon, Schulze and Whiteley \cite{Projective_Lens}.


Tay and Whiteley independently characterized which rank two incidence geometries have realizations as infinitesimally rigid body-and-hinge frameworks \cite{Tay1984, Whiteley88}. Furthermore, they jointly made the following conjecture \cite{TayWhi84}:

\begin{theorem} [Molecular conjecture]
	A rank two incidence geometry $S=(P,L,I)$ has a realization as an infinitesimally rigid body-and-hinge framework in $\mathbb{R}^d$ if and only if $S$ has a realization as an infinitesimally rigid panel-hinge framework in $\mathbb{R}^d$. 
	\label{molecular}
\end{theorem} 


The duals of panel-hinge frameworks in $\mathbb{R}^3$ provide a model for the motions of molecules. Therefore the molecular conjecture has implications for the rigidity of molecules. We refer to Whiteley for further explanation of this connection \cite{WhiteleyMolecular}.

In this article, we will focus on infinitesimal rigidity of rod configurations. Rod configurations are Euclidean realizations of rank two incidence geometries as points and straight line segments behaving as rigid bodies. The infinitesimal rigidity of rod configurations can be modeled by panel-hinge frameworks in the plane. Hence Theorem \ref{molecular} gives a way of determining the rigidity of sufficiently generic rod configurations that have two rods going through each point. Such rod configurations are realizations of rank two incidence geometries that are duals of graphs.

Whiteley proved the molecular conjecture for \textit{independent} body-and-hinge frameworks in the plane \cite{Whiteley89}. Whiteley's result also applies to independent body-and-joint frameworks, which realize rank two incidence geometries that are not duals of graphs. However, since it only applies to rank two incidence geometries that have realizations as independent body-and-joint frameworks, it does not imply the molecular conjecture, even in the plane.

Jackson and Jord\'an proved the molecular conjecture for body-and-hinge frameworks in the plane. They refer to such frameworks as body-and-pin frameworks \cite{JacJor08}. Katoh and Tanigawa proved the molecular conjecture in arbitrary dimension \cite{KatTan11}. 

There are other results in the literature concerning rigidity, points and lines that are not directly related to the rigidity of rod configurations. Jackson and Owen characterized the generic rigidity of planar frameworks consisting of points and lines in which distances between points and points, points and lines, and angles between lines can be fixed \cite{JacOwen2016}. They study the problem by representing the points and lines as vertices in a graph. The distances and angles that are preserved are represented as edges. In our setting, the distance between points and the lines they lie on is zero, which is a non-generic situation to which the tools developed by Jackson and Owen developed cannot be directly applied.

Jackson and Jord\'an characterize the frameworks of graphs in the plane that remain rigid when three given points are placed on a line \cite{JacJor05}. Determining which graph frameworks remain rigid when a set of points is placed on a line is different to the problem that we are interested in, so their result is not applicable in our situation. Jackson and Jord\'ans result was extended to an arbitrary number of points placed on a line by Eftekhari, Jackson, Nixon, Schulze, Tanigawa and Whiteley \cite{Frameworks2019}. 

Raz studied the combinatorial rigidity of graphs by considering the intersections of lines in $\mathbb{R}^3$ \cite{Raz2017}. Raz provided a model for rigidity of graph frameworks in the plane in terms of the incidence geometry of points and lines in $\mathbb{R}^3$. Her results concern rigidity of planar graph frameworks, not of rod configurations.

In this article we provide a way of determining the infinitesimal rigidity of rod configurations. We do this by relating the infinitesimal rigidity of a sufficiently generic rod configuration realizing a rank two incidence geometry to the generic rigidity of a graph associated to the incidence geometry. We call this graph the cone graph of the incidence geometry. 

Our result extends Whiteley's version of the molecular conjecture to body-and-joint frameworks that are not necessarily independent. In the special case that the incidence geometry defining the body-and-joint framework is the dual of a graph,  our result is essentially the planar case of Theorem \ref{molecular}.


\section{Preliminaries}

A \textit{framework} $\rho$ of a graph $G=(V,E)$ in $\mathbb{R}^d$ is an assignment of a point $\rho(v) \in \mathbb{R}^d$ to each vertex $v \in V$.

A \textit{continuous motion} of a framework of $G$ in $\mathbb{R}^d$ is a continuous motion of the points that preserves the distance between any two vertices $u$ and $v$ such that $(u,v) \in E$. The continuous Euclidean motions of $\mathbb{R}^d$ are continuous motions of any framework. We will refer to these motions as the \textit{trivial} motions of the framework. If the trivial motions are the only continuous motions of the framework, we say that the framework is (continuously) \textit{rigid}, otherwise we say that it is (continuously) \textit{flexible}.

An \textit{infinitesimal motion} of a framework $\rho$ of a graph $G$ in $\mathbb{R}^d$ is an assignment of a vector $m(v) \in \mathbb{R}^d$ to each vertex $v \in V$ such that $(\rho(v)-\rho(w))(m(v)-m(w))=0$ whenever $(v,w) \in E$. The linear part of a trivial motion gives a trivial infinitesimal motion. If the trivial infinitesimal motions are the only infinitesimal motions of the framework, then we say that the framework is \textit{infinitesimally rigid}, otherwise we say that it is \textit{infinitesimally flexible}. Furthermore, infinitesimal rigidity implies continuous rigidity \cite{Gluck}.

We say that a framework of a graph $G$ in $\mathbb{R}^d$ is \textit{generic} if the set of coordinates $\{\rho(v)\}_{v \in V}$ is algebraically independent over $\mathbb{Q}$. In general, continuous and infinitesimal rigidity depend on the framework but, by the following lemma, this is not the case for generic frameworks.

\begin{lemma}[\cite{Gluck,Graphs and Geometry}]
Let $G=(V,E)$ be a graph. If there is an infinitesimally rigid framework of $G$ in $\mathbb{R}^d$, then every generic framework of $G$ in $\mathbb{R}^d$ is rigid.
\label{GenericRigidity}
\end{lemma}

Furthermore, continuous rigidity and infinitesimal rigidity are equivalent for generic frameworks \cite{Roth}. It therefore makes sense to say that a graph $G$ is \textit{rigid} in $\mathbb{R}^d$ if every generic framework of $G$ in $\mathbb{R}^d$ is infinitesimally rigid or, equivalently, continuously rigid. Otherwise $G$ is flexible in $\mathbb{R}^d$.

A graph $G$ is \textit{minimally rigid} in $\mathbb{R}^d$ if $G=(V,E)$ is rigid in $\mathbb{R}^d$, but for any edge $e \in E$, the graph generated by $E \setminus \{e\}$ is flexible in $\mathbb{R}^d$. The following theorem gives a combinatorial characterization of minimally rigid graphs in the plane.

\begin{theorem}[Geiringer-Laman \cite{PolGei1927,Laman}]
	Let $G=(V,E)$ be a graph. Then $G$ is minimally rigid in $\mathbb{R}^2$ if and only if
	\begin{itemize}
		\item $|E|=2|V|-3$, and 
		\item $|E'| \leq 2|V'|-3$ for all subsets $E' \subseteq E$, where $V'$ is the support of $E'$.
	\end{itemize}
	\label{Geiringer-Laman}
\end{theorem}

We say that the set of edges $E$ of a graph $G=(V,E)$ is \textit{independent} if $|E'| \leq 2|V'|-3$ for all subsets $E' \subseteq E$, or, equivalently, if $E$ is independent in the generic two-dimensional rigidity matroid. A set of edges that is not independent is dependent. We refer to the book by Graver, Servatius and Servatius for further background on matroids and their use in combinatorial rigidity theory \cite{GraSerSer}. We say that a graph is independent if its set of edges is independent. 

We say that a subset of edges $E' \subseteq E$ is \textit{maximally independent in $E$} if $E'$ is independent, but $E' \cup \{e\}$ is dependent for any $e \in E \setminus E'$. We say that the subgraph $G'=(V',E')$ of $G=(V,E)$ generated by a maximally independent subset $E'$ of $E$ is \textit{maximally independent in $G$}. Note that an independent graph is a maximally independent subgraph of itself, and that a maximally independent subgraph of a rigid graph is minimally rigid.

A graph $G=(V,E)$ is a \textit{rigidity circuit} if its set of edges is minimally dependent, i.e. its set of edges is dependent, but $E \setminus \{e\}$ is independent for all edges $e \in E$. One example of a rigidity circuit is $K_4$. Clearly, all rigidity circuits satisfy $|E|=2|V|-2$, however not all graphs that satisfy $|E|=2|V|-2$ are rigidity circuits. See Figure \ref{Not_circuit} for an example.

Note that any dependent set of edges contains a minimally dependent set of edges, which generates a rigidity circuit. Another fact about rigidity circuits that we will use later on is that any vertex in a rigidity circuit has degree at least three.

\begin{figure}
	\begin{center}
	
	\begin{tikzpicture}
	\filldraw[black] (0,0) circle (2pt);
	\filldraw[black] (2,0) circle (2pt);
	\filldraw[black] (0,-1.5) circle (2pt);
	\filldraw[black] (2,-1.5) circle (2pt);
	\filldraw[black] (1,1) circle (2pt);
	
	\draw[thick, black] (0,0) -- (2,0) -- (2, -1.5) -- (0, -1.5) -- (0,0) -- (2, -1.5);
	\draw[thick, black] (0,0) -- (1,1) -- (2,0) -- (0, -1.5);
	\end{tikzpicture}
	
	\caption{A graph with $|V|=5$ and $|E|= 2|V|-2=8$ which is not a rigidity circuit.}	
	\label{Not_circuit}
	\end{center}
\end{figure}
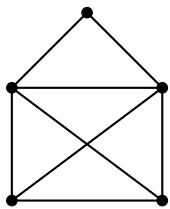

\subsection{Incidence geometries, rod and string configurations}

In the present paper, we are concerned with geometric realizations of rank two incidence geometries in $\mathbb{R}^2$. Rank two incidence geometries are also known as hypergraphs. We will only consider incidence geometries of rank two, so we will not explicitly state the rank of the incidence geometry. We will also assume that the incidence geometries we consider are connected. If the incidence geometry is not connected, then the realizations that we are interested in will be flexible, as they have two components which can move independently.

Let $S=(P,L,I)$ be an incidence geometry. A \textit{linear realization} $\rho$ of $S$ in $\mathbb{R}^2$ is an assignment of a line in $\mathbb{R}^2$, with slope $f_i$ and $y$-intercept $h_i$, to each element $\ell_i \in L$, and an assignment of point coordinates $(x_j,y_j)$ to each element $p_j \in P$ such that if $(p_j,\ell_i) \in I$, then 

\begin{equation}
f_ix_j + y_j + h_i =0.
\label{redrawings}
\end{equation}

A linear realization is \textit{proper} if distinct points are assigned distinct coordinates, and \textit{trivial} if all points are assigned the same coordinates. Not all incidence geometries have proper linear realizations in $\mathbb{R}^2$. For example, it is well-known that the Fano plane does not.

Given a linear realization $\rho$ of $S=(P,L,I)$ in the plane, we can fix its line slopes $\{f_i\}$, and solve the $I$ equations of the form (\ref{redrawings}) for $x_j$, $y_j$ and $h_i$, to determine a space of linear realizations of $S$ with the same line slopes as the linear realization $\rho$. Call this space the space of \textit{parallel redrawings} of the linear realization $\rho$. We call the $|I| \times (|L| + 2|P|)$-matrix encoding the linear equations of the form $(\ref{redrawings})$ the \textit{concurrence geometry matrix}, and denote it by $M(S, \rho)$. 

The translations and the dilation generate a three-dimensional space of parallel redrawings of any linear realization $\rho$ of an incidence geometry in the plane such that at least two points are realized with distinct coordinates. Therefore, the maximum rank of $M(S, \rho)$ is $|L|+2|P|-3$, as the kernel of $M(S, \rho)$ is at least three-dimensional. The concurrence geometry matrix has $|I|$ rows, so if the rows of $M(S, \rho)$ are independent for a proper linear realization $\rho$, then any subset $J \subseteq I$ with support $Q \times M \subseteq P \times L$ must satisfy $|J| \leq |M| + 2|Q|-3$. See \cite{SerWhi} for more about parallel redrawings and how they can be applied.

Let $S=(P,L,I)$ be an incidence geometry. We say that a subset $I' \subseteq I$ is \textit{sharply independent} if 
\begin{equation}
|J| \leq |M|+2|Q|-3
\label{inequality}
\end{equation}
for any subset $J \subseteq I'$, where $Q \times M \subseteq P \times L$ is the support of $J$.

A \textit{rod configuration} realizing an incidence geometry $S=(P,L,I)$ is a linear realization of $S$ where each line is considered to be an infinitesimally rigid body.

More specifically, an infinitesimal motion of a rod configuration realizing an incidence geometry $S=(P,L,I)$ in $\mathbb{R}^2$ is an assignment of a vector $v \in \mathbb{R}^2$ to each point $p \in P$ so that restricted to each rod $\ell \in L$, the vectors are linearizations of rotations or translations of the rod. We explore infinitesimal rigidity of rod configurations further in another article \cite{rods}.

A \textit{string configuration} realizing an incidence geometry $S=(P,L,I)$ in $\mathbb{R}^2$ is given by a framework in $\mathbb{R}^2$ of a graph, with vertex set $V=P$ and such that for every line $\ell \in L$ there is a tree on the vertices representing points incident to $\ell$, so that the tree on the vertices representing points incident to $\ell$ is collinear. 

Note that a string configuration of $S$ requires a line slope assigned to every line in $L$, and point coordinates for every point in $P$, i.e. a linear realization of $S$. We can therefore for example say that a string configuration is proper if the underlying linear realization is proper. 

String configurations come with a notion of rigidity, namely the notion of rigidity for frameworks of graphs. The infinitesimal motions of a string configuration in the plane are in one-to-one correspondence with its parallel redrawings, see \cite{Crapo84,Crapo85,Whiteley88}. Using this fact, Whiteley proved the following theorem:

\begin{theorem}[Whiteley \cite{Whiteley89}]
Let $S=(P,L,I)$ be an incidence geometry. Then $S$ has a realization as a minimally infinitesimally rigid string configuration in the plane if and only if
\begin{itemize}
	\item $I$ is sharply independent, and
	\item $|I| = |L|+2|P|-3$.
\end{itemize}
\label{Whiteley5.2}
\end{theorem}

In fact, any string configuration realizing $S$ in $\mathbb{R}^2$ with an underlying linear realization $\rho$ that maximizes the rank of the concurrence geometry matrix $M(S, \rho)$ is infinitesimally rigid.  

\begin{remark}
In this setting the infinitesimal rigidity is determined only by the line slopes, so the choice of tree on the vertices representing the points incident to a line $\ell$ does not affect the rigidity of the string configuration. 
\end{remark}

\subsection{Body-and-joint frameworks and their rigidity properties}

Let $S=(P,L,I)$ be an incidence geometry. A planar \textit{body-and-joint framework} consists of an incidence geometry $S$ in $\mathbb{R}^2$ and an assignment of a point in $\mathbb{R}^2$ to every point in $P$. We say that the body-and-joint framework realizes $S$. A body-and-joint framework in $\mathbb{R}^2$ is \textit{infinitesimally rigid} if each line can be replaced by a minimally rigid framework of a graph with vertex set including the points incident to the line so that the entire framework is infinitesimally rigid, as a graph framework. 

Similarly, a body-and-joint framework realizing $S$ in $\mathbb{R}^2$ is \textit{minimally infinitesimally rigid} if it is possible to replace each line by a minimally rigid framework of a graph with vertex set including the points incident to the line so that the entire framework is minimally infinitesimally rigid.

A body-and-joint framework realizing $S$ in $\mathbb{R}^2$ is \textit{independent} if each line can be replaced by a minimally rigid graph with vertex set including the points incident to the line so that the resulting graph is independent. 

Whiteley proved the following result characterizing which incidence geometries have realizations as minimally infinitesimally rigid body-and-joint frameworks in $\mathbb{R}^2$. 

\begin{theorem}[Whiteley \cite{Whiteley89}]
Given an incidence geometry $S=(P,L,I)$ the following are equivalent:
	\begin{enumerate}
	\item $S$ has an independent (minimally infinitesimally rigid) body and joint realization in $\mathbb{R}^2$.
	\item $S$ satisfies $2|I| \leq 3|L| + 2|P| -3$, ($2|I| = 3|L| + 2|P| -3$) and for every subset of bodies with the induced subgraph of attached joints $2|I'| \leq 3|L'|+2|P'|-3$.
	\item $S$ has an independent (minimally infinitesimally rigid) body and joint realization in $\mathbb{R}^2$ such that each body has all its joints collinear.
	\end{enumerate}
	\label{Whiteley5.4}
\end{theorem}

The equivalence $1 \iff 3$ implies Theorem \ref{molecular} in the special case that the incidence geometry has a realization as a minimally infinitesimally rigid body-and-hinge framework, but Theorem \ref{Whiteley5.4} also applies to incidence geometries where more than two lines meet at a point.

We will now introduce some particular body-and-joint frameworks of an incidence geometry. Our goal is to relate the infinitesimal rigidity of rod configurations realizing $S$ to the rigidity of these body-and-joint frameworks.

Let $S=(P,L,I)$ be an incidence geometry. A \textit{cone graph} $G^C(S)$ of $S$ is a graph with a vertex for each point in $P$ and a vertex for each point in $L$, and the set of edges in $G^C(S)$ consists of edges from the vertex representing a line $\ell$ to all vertices representing points incident to $\ell$, and edges $\{(p,q)\}_{(q,\ell) \in I}$, where $p$ is some chosen vertex incident to $\ell$. We call the vertices that represent points in $P$ \textit{point vertices}, and the vertices representing lines in $L$ \textit{cone vertices}. We will refer to the subgraph of $G^C(S)$ generated by a cone vertex representing a line $\ell$ and the vertices representing the points $\{p_1, ... , p_k\}$ such that $(p_j, \ell) \in I$ as the \textit{cone of $\ell$}.

Note that the subgraph generated by point vertices representing points incident to the same line $\ell$ is a \textit{star graph}. We will refer to the vertex $p$ as the \textit{inner vertex} of the star graph. The cone of $\ell$ depends on the choice of $p$, so in general there are many cone graphs of an incidence geometry $S$. See Figure \ref{Cone_graph_rods} for an example of a cone graph of a rod configuration. 

\begin{figure}
	\begin{center}
	\begin{tabular}{cc}
	\begin{tikzpicture}
\filldraw[black] (0,0) circle (2pt);
\filldraw[black] (-1.5,-3) circle (2pt);
\filldraw[black] (1.5,-3) circle (2pt);
\filldraw[black] (0.5,-1) circle (2pt);
\filldraw[black] (-0.5,-1) circle (2pt);
\filldraw[black] (0,-3) circle (2pt);
\filldraw[black] (0,-1.5) circle (2pt);


\draw[thick] (0,0) -- (1.5,-3) -- (-1.5,-3)--(0,0);
\draw[thick] (-1.5,-3) -- (0.5,-1);




\end{tikzpicture}
	&
	\begin{tikzpicture}
\filldraw[black] (0,0) circle (2pt);
\filldraw[black] (-1.5,-3) circle (2pt);
\filldraw[black] (1.5,-3) circle (2pt);
\filldraw[black] (0.3,-1.1) circle (2pt);
\filldraw[black] (-0.25,-1.5) circle (2pt);
\filldraw[black] (0.3,-2.75) circle (2pt);
\filldraw[black] (0.2,-1.75) circle (2pt);

\filldraw[black] (1.5,-0.8) circle (2pt);
\filldraw[black] (-1.4,-1) circle (2pt);
\filldraw[black] (-0.4,-3.5) circle (2pt);
\filldraw[black] (0.5,-2) circle (2pt);

\draw[thick] (0,0) -- (0.25, -1.1);
\draw[thick] (0.25,-1.1)--(1.5,-3);

\draw[thick] (0,0) -- (-0.25,-1.5);
\draw[thick] (-0.25,-1.5)--(-1.5,-3);

\draw[thick] (0.3,-1.1) -- (0.2,-1.75);
\draw[thick] (-1.5,-3) -- (0.2,-1.75);

\draw[thick] (1.5,-0.8)--(0,0);
\draw[thick] (1.5,-0.8)--(0.3,-1.1);
\draw[thick] (1.5,-0.8)--(1.5,-3);

\draw[thick] (-1.4,-1)--(0,0);
\draw[thick] (-1.4,-1)--(-0.25,-1.5);
\draw[thick] (-1.4,-1)--(-1.5,-3);

\draw[thick] (0.5, -2) -- (-1.5,-3);
\draw[thick] (0.5, -2) -- (0.3,-1.1);
\draw[thick] (0.5, -2) -- (0.2,-1.75);

\draw[thick] (1.5,-3) -- (0.3,-2.75);
\draw[thick] (-1.4,-3) -- (0.3,-2.75);

\draw[thick] (-0.4,-3.5) -- (0.3,-2.75);
\draw[thick] (-0.4,-3.5) -- (1.5,-3);
\draw[thick] (-0.4,-3.5) -- (-1.5,-3);
\end{tikzpicture}
	\end{tabular}
	\caption{A rod configuration realizing an incidence geometry $S$ and a cone graph $G^C(S)$ of $S$.}	
	\label{Cone_graph_rods}
	\end{center}
\end{figure}
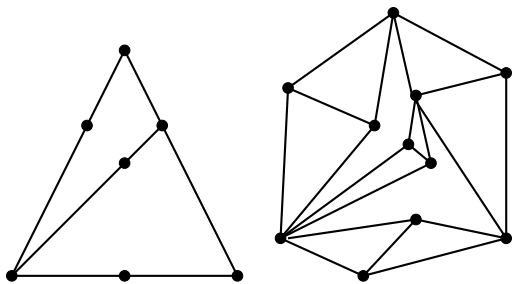

Given an incidence geometry $S=(P,L,I)$ define the \textit{cone incidence geometry} $S^C=(P^C,L^C,I^C)$ so that there is a point in $P^C$ for each point in $P$ and for each line in $L$, the lines in $L^C$ are the subsets of points that are collinear in $L$, and for any line $\ell \in L$ there is a line $\langle c_\ell,p \rangle$ incident to the cone vertex $c_\ell$ representing $\ell$ and any point $p$ incident to $\ell$.

For any incidence geometry $S$, there is a unique cone incidence geometry $S^C$. A string configuration realizing $S^C$ in $\mathbb{R}^2$ with the collinear tree on the point vertices representing points incident to a line $\ell$ chosen to be a star graph is a framework of a cone graph of $S$ in special position. See Figure \ref{Cone_graph_ng} for an example.

Linear realizations of $S^C$ give rise to rod configurations realizing $S$, and conversely. Furthermore, a string configuration realizing $S^C$ is infinitesimally rigid if and only if the rod configuration realizing $S$ that it gives rise to is infinitesimally rigid. Therefore analyzing the infinitesimal rigidity of string configurations realizing $S^C$ is of interest to us.

\begin{figure}
	\begin{center}
	\begin{tabular}{cc}
	\begin{tikzpicture}
\filldraw[black] (0,0) circle (2pt);
\filldraw[black] (-1.5,-3) circle (2pt);
\filldraw[black] (1.5,-3) circle (2pt);
\filldraw[black] (0.3,-1.1) circle (2pt);
\filldraw[black] (-0.25,-1.5) circle (2pt);
\filldraw[black] (0.3,-2.75) circle (2pt);
\filldraw[black] (0.2,-1.75) circle (2pt);

\filldraw[black] (1.5,-0.8) circle (2pt);
\filldraw[black] (-1.4,-1) circle (2pt);
\filldraw[black] (-0.4,-3.5) circle (2pt);
\filldraw[black] (0.5,-2) circle (2pt);

\draw[thick] (0,0) -- (0.25, -1.1);
\draw[thick] (0.25,-1.1)--(1.5,-3);

\draw[thick] (0,0) -- (-0.25,-1.5);
\draw[thick] (-0.25,-1.5)--(-1.5,-3);

\draw[thick] (0.3,-1.1) -- (0.2,-1.75);
\draw[thick] (-1.5,-3) -- (0.2,-1.75);

\draw[thick] (1.5,-0.8)--(0,0);
\draw[thick] (1.5,-0.8)--(0.3,-1.1);
\draw[thick] (1.5,-0.8)--(1.5,-3);

\draw[thick] (-1.4,-1)--(0,0);
\draw[thick] (-1.4,-1)--(-0.25,-1.5);
\draw[thick] (-1.4,-1)--(-1.5,-3);

\draw[thick] (0.5, -2) -- (-1.5,-3);
\draw[thick] (0.5, -2) -- (0.3,-1.1);
\draw[thick] (0.5, -2) -- (0.2,-1.75);

\draw[thick] (1.5,-3) -- (0.3,-2.75);
\draw[thick] (-1.4,-3) -- (0.3,-2.75);

\draw[thick] (-0.4,-3.5) -- (0.3,-2.75);
\draw[thick] (-0.4,-3.5) -- (1.5,-3);
\draw[thick] (-0.4,-3.5) -- (-1.5,-3);
\end{tikzpicture}
&
	\begin{tikzpicture}
\filldraw[black] (0,0) circle (2pt);
\filldraw[black] (-1.5,-3) circle (2pt);
\filldraw[black] (1.5,-3) circle (2pt);
\filldraw[black] (0.5,-1) circle (2pt);
\filldraw[black] (-0.5,-1) circle (2pt);
\filldraw[black] (0,-3) circle (2pt);
\filldraw[black] (0,-1.5) circle (2pt);

\filldraw[black] (1.5,-1) circle (2pt);
\filldraw[black] (-1.5,-1) circle (2pt);
\filldraw[black] (0,-3.5) circle (2pt);
\filldraw[black] (0.5,-2) circle (2pt);

\draw[thick] (0,0) -- (1.5,-3) -- (-1.5,-3)--(0,0);
\draw[thick] (-1.5,-3) -- (0.5,-1);

\draw[thick] (1.5,-1)--(0,0);
\draw[thick] (1.5,-1)--(0.5,-1);
\draw[thick] (1.5,-1)--(1.5,-3);

\draw[thick] (-1.5,-1)--(0,0);
\draw[thick] (-1.5,-1)--(-0.5,-1);
\draw[thick] (-1.5,-1)--(-1.5,-3);

\draw[thick] (0.5, -2) -- (-1.5,-3);
\draw[thick] (0.5, -2) -- (0.5,-1);
\draw[thick] (0.5, -2) -- (0,-1.5);

\draw[thick] (0,-3.5) -- (0,-3);
\draw[thick] (0,-3.5) -- (1.5,-3);
\draw[thick] (0,-3.5) -- (-1.5,-3);
\end{tikzpicture}
	\end{tabular}
	\caption{The cone graph $G^C(S)$ from Figure \ref{Cone_graph_rods} and a string configuration realizing $S^C$.}	
		\label{Cone_graph_ng}
	\end{center}
\end{figure}
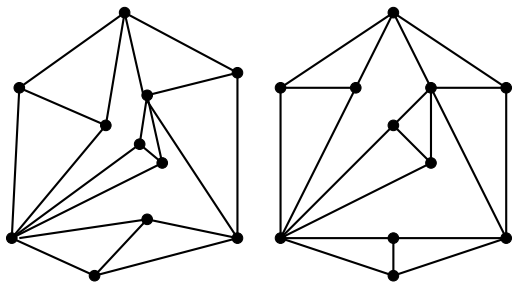


Whiteley proved the implication $2 \implies 3$ of Theorem \ref{Whiteley5.4} by showing that if the incidence geometry satisfies the count given in 2, then the cone incidence geometry has a realization as a minimally infinitesimally rigid string configuration in $\mathbb{R}^2$. Such a string configuration is a minimally infinitesimally rigid body-and-joint framework realizing the incidence geometry, with all joints incident to a body collinear.

\section{Some useful lemmas}


In this section, we will state and prove four lemmas that we will use to prove our main result. For simplicity of notation, we will in this section denote edge sets by $E^*$ where $*$ is some suitable decoration, and the vertex set generated by the edge set $E^*$ by $V^*$, e.g. the set $V^i$ is the set of vertices generated by $E^i$.

First, recall that there is not a unique cone graph of an incidence geometry, since there are several choices for the inner vertex for each line, and in general each choice yields a different graph. The choice of inner vertex, however, does not affect the rigidity properties, as we show in the following lemma.

\begin{lemma}
Let $S=(P,L,I)$ be an incidence geometry. If one cone graph of $S$ is rigid in $\mathbb{R}^2$, then all cone graphs of $S$ are rigid in $\mathbb{R}^2$.
\label{RigidConeGraph}
\end{lemma}

\begin{proof}
Suppose $G^C(S)$ is a cone graph of $S$ which is rigid in $\mathbb{R}^2$, and $G_1^C(S)$ is some other cone graph of $S$. If the cone graphs $G^C(S)$ and $G_1^C(S)$ are different, then the inner vertices of the star graph on the point vertices will be different for some line(s).

Take such a line $\ell$. Suppose $p$ is the inner vertex  in $G^C(S)$ of the star graph representing a line $\ell \in L$, and $p'$ is the inner vertex in $G_1^C(S)$ of the star graph on representing the same line $\ell$. Let $p, p', p_1, p_2,...,p_k$ be the point vertices in the cone of $\ell$. There are edges $(p,p_i)$, for $1\leq i \leq k$, in $G^C(S)$, and edges $(p',p_i)$, for $1 \leq i \leq k$ in $G_1^C(S)$. The edge $(p,p')$ is in both graphs. 

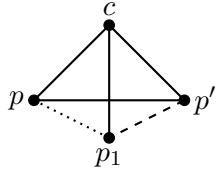
\begin{figure}
	\begin{center}
	
	\begin{tikzpicture}
	\filldraw[black] (0,0) circle (2pt);
	\filldraw[black] (-1,-1) circle (2pt);
	\filldraw[black] (0,-1.5) circle (2pt);
	\filldraw[black] (1,-1) circle (2pt);

	\draw[thick] (0,0) -- (-1,-1);
	\draw[thick] (0,0) -- (0,-1.5);
	\draw[thick] (0,0) -- (1,-1);
	\draw[thick] (-1,-1) -- (1,-1);
	\draw[thick, dotted] (-1,-1) -- (0,-1.5);
	\draw[thick, dashed] (1,-1) -- (0,-1.5);  

	\node[anchor=south] at (0,0) {$c$};
	\node[anchor=west] at (1,-1) {$p'$};
	\node[anchor=east] at (-1,-1) {$p$};
	\node[anchor=north] at (0, -1.5) {$p_1$};

	\end{tikzpicture}
	
	\caption{It is possible to switch certain edges in a cone without making the cone graph flexible.}	
	\label{switch}
	\end{center}
\end{figure}

Adding the edge $(p',p_1)$ to the edge set of $G^C(S)$ creates a copy of $K_4$ on the vertices $p$, $p'$, $p_1$ and $c$, where $c$ is the cone vertex representing $\ell$. See Figure \ref{switch}. $K_4$ is a rigidity circuit, and as the edge $(p, p_1)$ is one of the edges in the copy of $K_4$ that is created, we can subsequently remove the edge $(p,p_1)$ without the graph becoming flexible. In the same way we can replace all edges $(p, p_i)$ in the cone of $\ell$ by edges $(p', p_i)$ while preserving rigidity, to obtain a rigid cone graph where $p'$ is the inner vertex of the star subgraph on the point vertices in the cone of $\ell$.

This process can be repeated for any line $\ell$ such that the inner vertex of the cone of $\ell$ in $G_1^C(S)$ is different from the inner vertex of the cone of $\ell$ in $G^C(S)$, until we obtain $G_1^C(S)$. Hence $G_1^C(S)$ is also rigid in $\mathbb{R}^2$.
\end{proof}

We will now state and prove two lemmas concerning minimally rigid subgraphs of a graph $G=(V,E)$ with vertex sets that intersect non-trivially. 

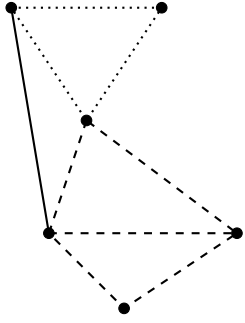
\begin{figure}
	\begin{center}
	
	\begin{tikzpicture}
	\filldraw[black] (0,0) circle (2pt);
	\filldraw[black] (2,0) circle (2pt);
	\filldraw[black] (1,-1.5) circle (2pt);
	\filldraw[black] (0.5,-3) circle (2pt);
	\filldraw[black] (3,-3) circle (2pt);
	\filldraw[black] (1.5,-4) circle (2pt);
	
	\draw[thick, dotted] (0,0) -- (2,0)--(1,-1.5)--(0,0);
	\draw[thick, dashed] (1,-1.5) -- (0.5,-3)--(1.5,-4)--(3,-3)--(1,-1.5);
	\draw[thick, dashed] (0.5,-3) -- (3,-3);
	\draw[thick] (0,0) -- (0.5,-3);
	\end{tikzpicture}
	
	\caption{Two minimally rigid graphs that intersect in one vertex, and an added edge that makes the graph rigid.}	
	\label{OnePoint}
	\end{center}
\end{figure}

The next lemma concerns minimally rigid subgraphs of a graph with vertex sets that intersect in exactly one vertex. Figure \ref{OnePoint} shows an example of two minimally rigid graphs, a triangle (dotted) and a minimally rigid graph on four vertices (dashed), that intersect in a single vertex. Clearly, the graph consisting of the dotted and dashed edges in Figure \ref{OnePoint} is flexible. However, adding the solid edge as in the figure would make the graph minimally rigid. The lemma says that the situation illustrated by the graph in Figure \ref{OnePoint} is the only situation that can occur if two minimally rigid subgraphs of a graph $G$ intersect in a point.

\begin{lemma}
Let $G_1=(V_1,E_1)$ and $G_2=(V_2,E_2)$ be subgraphs of a graph $G=(V,E)$ such that $G_1$ and $G_2$ are minimally rigid in $\mathbb{R}^2$, and $V_1 \cap V_2=\{v\}$. Let $E_3=E_1 \cup E_2$. Also, let $e=(v_1, v_2)$ be an arbitrary edge with $v_1 \in V_1 \setminus \{v\}$ and $v_2 \in V_2 \setminus \{v\}$, and let $E_4=E_3 \cup \{e\}$. Then

\begin{enumerate}
	\item the graph $G_3=(V_3, E_3)$ generated by $E_3$ is flexible in $\mathbb{R}^2$, and
	\item the graph $G_4=(V_4,E_4)$ generated by $E_4$ is minimally rigid in $\mathbb{R}^2$.
\end{enumerate}
\label{empty_intersect}
\end{lemma}

\begin{proof}
\textit{1.} Note that the intersection of $E_1$ and $E_2$ must be empty, as $V_1 \cap V_2=\{v\}$. $G_1$ and $G_2$ are minimally rigid, so $|E_3|=|E_1|+|E_2|=2|V_1|-3+2|V_2|-3=2(|V_1|+|V_2|)-6$ by Theorem \ref{Geiringer-Laman}. Also, $|V_3|=|V_1|+|V_2|-1$, so $|E_3|=2|V_1|+2|V_2|-6=2|V_3|-4$. Since $|E_3| < 2|V_3|-3$, $G_3$ is flexible.

\textit{2.} We will use Theorem \ref{Geiringer-Laman} to show that $G_4$ is minimally rigid in $\mathbb{R}^2$. Let $E' \subseteq E_4$ be an arbitrary subset. Suppose first that $e \notin E'$. Then $E' \subseteq E_3= E_1 \cup E_2$. If $E' \subseteq E_1$ or $E' \subseteq E_2$, then $|E'| \leq 2|V'| - 3$ as $G_1$ and $G_2$ are minimally rigid in $\mathbb{R}^2$. 

 Else, we can write $E'$ as $E' = E_1' \cup E_2'$, where $E_1' \subseteq E_1$ and $E_2' \subseteq E_2$, then $|E'|=|E_1'|+|E_2'| \leq 2|V_1'| -3 + 2|V_2'| -3$. Recall that we assumed $V_1 \cap V_2 = \{ v \}$, so if $v \notin V'$, then $|V'|=|V_1'|+|V_2'|$ and $|E'| \leq 2(|V_1'|+|V_2'|)-6= 2|V'|-6$. 

If $v \in V'$, then $|V'|=|V_1'|+|V_2'|-1$, and $|E'| \leq 2(|V_1|+|V_2|)-6=2|V'|-4$.

Suppose instead that $e \in E'$. Let $E''=E' \setminus \{e\}$, and $V''$ be the support of $E''$. Note that $E'' \subseteq E_3$. If $v_1 \in V''$ and $v_2 \in V''$, then $E''$ cannot be strictly contained in either $E_1$ or $E_2$, since $v_1 \in V_1 \setminus \{v\}$ and $v_2 \in V_2 \setminus \{v\}$. We have seen that $|E''| \leq 2|V''|-4$ in this case. Furthermore $|V''|=|V'|$, so $|E'|=|E''|+1 \leq 2|V''|-3 = 2|V'|-3$. 

If one of $v_1$ and $v_2$, but not both, is in $V''$, then $|V''|=|V'|-1$ and $|E'|=|E''|+1 \leq 2|V''|-2 = 2|V'|-4$, where $|E''| \leq 2|V''| - 3$ since $E'' \subseteq E_3$. 

If neither of $v_1$ and $v_2$ is in $V''$, then $|V''|=|V'|-2$ and $|E'|=|E''|+1 \leq 2|V''|-2=2|V'|-6$. 

Hence, $|E'| \leq 2|V'| - 3$ for all subsets $E' \subseteq E_4$. We also showed that $|E_3|=2|V_3|-4$, and since $|E_4|=|E_3|+1$ and $|V_4|=|V_3|$, it follows that $|E_4|=2|V_3|-3$. Therefore $G_4$ is minimally rigid in $\mathbb{R}^2$ by Theorem \ref{Geiringer-Laman}.
\end{proof}

\begin{figure}
	\begin{center}
	
	\begin{tikzpicture}
	\filldraw[black] (0,0) circle (2pt);
	\filldraw[black] (-1,-1.5) circle (2pt);
	\filldraw[black] (1,-1.5) circle (2pt);
	\filldraw[black] (-0.5,-3) circle (2pt);
	\filldraw[black] (2,-3) circle (2pt);
	
	\draw[thick, dotted] (0,0) -- (-1,-1.5);
	\draw[thick] (1,-1.5) -- (-1,-1.5);
	\draw[thick, dotted] (1,-1.5)--(0,0);
	\draw[thick, dashed] (-1,-1.5)--(-0.5,-3)--(2,-3)--(1,-1.5)--(-0.5,-3);
	\end{tikzpicture}
	
	\caption{Two minimally rigid graphs that intersect in an edge and two points.}
	\label{OneEdge}
	\end{center}
\end{figure}
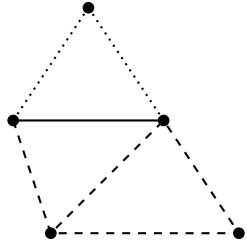

Next, we will prove a lemma describing what happens if the graph $G=(V,E)$ is independent, and the minimally rigid graphs intersect in more than one point. Figure \ref{OneEdge} shows a triangle (dotted and solid edges) and a minimally rigid graph on four vertices (dashed and solid edges) with vertex sets that intersect in two vertices. We can see that the edge sets also intersect in the solid edge. We will see that if $G$ is independent and the vertex sets of the minimally rigid subgraphs intersect in more than one vertex, then the edge sets always intersect. The graph in Figure \ref{OneEdge} is minimally rigid.

\begin{lemma}
Let $G_1=(V_1,E_1)$ and $G_2=(V_2, E_2)$ be subgraphs of an independent graph $G=(V,E)$ that are minimally rigid in $\mathbb{R}^2$. Let $E_3=E_1 \cup E_2$. If $|V_1 \cap V_2| > 1$, then 
 	\begin{enumerate}
	\item $|E_1 \cap E_2| \geq 1$ and 
	\item the graph $G_3=(V_3, E_3)$ generated by $E_3$ is minimally rigid.
	\end{enumerate}
	\label{intersection}
\end{lemma}

\begin{proof}
	\textit{1.} Suppose $|V_1 \cap V_2|=k \geq 2$, but $|E_1 \cap E_2|=0$. Then $|V_3| \leq |V_1|+|V_2| - 2$, and $|E_3|=|E_1| + |E_2| = 2|V_1|+2|V_2| - 6 = 2(|V_1|+|V_2|-2)-2 \geq 2|V_3|-2$. This contradicts $G=(V,E)$ being independent. 
	
	\textit{2.} Let $|E_1 \cap E_2| = i \geq 1$. Then $|V_1 \cap V_2| \geq i+1 \geq 2$, so $|V_3| \leq |V_1|+|V_2|-(i+1) \leq |V_1| + |V_2| - 2$, and $|E_3|=|E_1|+|E_2|-i = 2|V_1|+2|V_2| -6 - i = 2(|V_1|+|V_2|)-(i+1)-2-3 \geq 2(|V_1|+|V_2|) - 2(i+1) -3 \geq 2|V_3|-3$. Since $G_3$ is a subgraph of an independent graph $G$, $G_3$ is also independent, and by Theorem \ref{Geiringer-Laman} it must also hold that $|E_3| \leq 2|V_3|-3$, so we conclude that $|E_3| \geq 2|V_3|-3$, and consequently $|E_3|=2|V_3|-3$. As $G_3$ is independent, and $|E_3|=2|V_3|-3$, $G_3$ must be minimally rigid in $\mathbb{R}^2$, by Theorem \ref{Geiringer-Laman}. 
\end{proof}

In the following lemma, we construct a subgraph of $G^C(S)$ and a subgeometry of $S^C$ that will allow us to relate the rigidity properties of $G^C(S)$ and $S^C$.

\begin{lemma}
Let $S=(P,L,I)$ be an incidence geometry, and let $G'=(V',E')$ be a subgraph of a cone graph $G^C(S)=(V^C,E^C)$ of $S$ with $V=V^C$. Define the subgeometry $S'=(P',L',I')$ of $S^C=(P^C,L^C,I^C)$ such that $P'=P^C$, $L'$ consists of the subsets of $P^C$ that are collinear in $L^C$ and lie in a common star subgraph in $G'$, and $I'$ is defined by inclusion. Then there is a subgraph $G'=(V',E')$ of a cone graph $G^C(S)=(V^C,E^C)$ of $S$ with $V'=V^C$ such that
\begin{enumerate}
	\item $G'$ is maximally independent in $G^C(S)$,
	\item the subgeometry $S'=(P',L',I')$ of $S^C=(P^C,L^C,I^C)$, is sharply independent, and
	\item $S'=(P',L',I')$ satisfies $|I'|=|L'|+2|P'|-3$ if $G'$ satisfies $|E'|=2|V'|-3$.
\end{enumerate}
\label{Construction}
\end{lemma}

\begin{proof}
First, we will construct the subgraph $G'$ and prove that it is maximally independent. Since $S$ is connected by assumption, we can order the lines in $L=\{\ell_1, \ell_2,...,\ell_n\}$ so that each line $\ell_i$ has at least one point in common with the set of lines $\{\ell_j\}_{j<i}$. We construct $G'$ recursively, by constructing a maximally independent subgraph $G_i$ of $G^C(S_i)$, where $S_i$ is the incidence geometry generated by $L_i=\{\ell_j \}_{j \leq i}$. The graph $G'$ we are looking for is $G_n$.

Let $G_1=(V_1,E_1)$ be a cone of the line $\ell_1$. Suppose that $G_{i-1}$ is given, and has as its vertices the points belonging to the set of lines $L_{i-1}$. By construction, we know that $\ell_i$ has at least one point $p$ in common with the lines in $L_{i-1}$. The point $p$ then corresponds to a vertex $p$ in the graph $G_{i-1}$. There may also be other points $q_1$,$q_2$,...,$q_l$ incident to $\ell_i$ that correspond to vertices in $G_{i-1}$. Further, there may be points $p_1$, $p_2$,...,$p_m$ incident to $\ell_i$, but not to any of the lines in $L_{i-1}$, that therefore do not correspond to vertices in $G_{i-1}$. 

To construct $G_i=(V_i, E_i)$, we perform the following steps, in order:
 
\begin{enumerate}
\item add the cone vertex $c$ corresponding to $\ell_{i}$ and vertices corresponding to $p_1, ..., p_m$ to $V_{i-1}$, i.e. let $V_i=V_{i-1} \cup \{c, p_1, ... , p_m\}$;
\item add the edge $(c,p)$ to $E_{i-1}$, i.e. set $E_i=E_{i-1} \cup \{(c,p)\}$; 
\item if $l \geq 1$, add the edge $(c, q_1)$ to $E_i$;
\item add edges $(c,p_i)$ for $1 \leq i \leq m$ to $E_i$;
\item add edges $(p,p_i)$ for $1 \leq i \leq m$ to $E_i$;
\item add all edges $(c,q_i)$ that do not make $E_i$ dependent;
\item add all edges $(p, q_i)$ such that $(c, q_i)$ was added in the previous step, and adding $(p,q_i)$ does not make $E_i$ dependent.
\end{enumerate}

Note that by construction $V_n=V'=V^C$. Note also that the edge $(p, q_i)$ is not added to $E_i$ in the final step unless the edge $(c,q_i)$ was added to $E_i$ first.

To show that $G'=G_n$ is independent in $G^C(S)$, we show by induction on the number of lines that each $G_i$ is maximally independent in $G^C(S_i)$. The first graph $G_1$ is a cone over a star, which is minimally rigid, and therefore $G_1$ is maximally independent in $G^C(S_1)$. 

Suppose that $G_{i-1}$ is maximally independent in $G^C(S_{i-1})$. 

If $l \geq 1$, then adding the cone vertex $c$ corresponding to $\ell_i$ and the edges $(p,c)$ and $(q_i,c)$ preserves independence, since one vertex and two edges are added. 

If $l=0$, then adding the cone vertex and the edge $(c,p)$ also preserves independence, as one vertex and one edge are added. 

Similarly, adding the points $p_i$ and the edges $(p,p_i)$ and $(c, p_i)$ preserves independence. Furthermore, the edges $(c,q_i)$ and $(p,q_i)$ are only added if adding them preserves independence among the edges. 

Hence $G_i$ is independent, since $G_{i-1}$ is independent and each step in the construction preserves independence. So we have shown that $G'=G_n$ is independent.

We will now show that $G_i$ is \textit{maximally} independent in $G^C(S_i)$. By assumption, $G_{i-1}$ is maximally independent in $G^C(S_{i-1})$, so we cannot add any more edges to the cones of the lines $\ell_1, ..., \ell_{i-1}$ while preserving independence. By construction, all edges $(c,p_k)$, $(p, p_k)$, and the edge $(c,p)$, are in $G_i$.
Also by construction, we cannot add any more edges $(c, q_i)$, or edges $(p,q_i)$ such that the edge $(c,q_i)$ is in $E_i$ without the edge set becoming dependent. 

We claim that adding any of the remaining edges, namely the edges $(p,q_i)$ such that the edge $(c,q_i)$ is not added in our construction, would make the set of edges $E_i$ dependent. 

Suppose for a contradiction that there is some edge $e=(p,q_i)$ such that $f=(c,q_i)$ is not in $E_i$, but $E_i \cup \{e\}$ is independent. If $f$ is not in $E_i$, then $E_i \cup \{f\}$ is dependent. Hence, there is a rigidity circuit containing $f$ generated by some subset of edges $E^1 \subseteq E_i \cup \{f\}$. As $E^1$ generates a rigidity circuit, $|E^1|=2|V^1|-2$. Let $E^2=E^1 \setminus \{f\} \cup \{e\}$. If $E_i \cup \{e\}$ is independent, then so is $E^2$, since $E^2 \subseteq E_i \cup \{e\}$. 

Every vertex in a rigidity circuit must have degree at least 3, so $c$ must be a vertex in the graph generated by $E^1 \setminus \{f\}$. Also, $E^1 \setminus \{f\} \subset E^2$, so $c \in V^2$. By construction $|E^1|= |E^2|$. 

\textit{Case 1.} If $p \in V^1$, then $|V^1|=|V^2|$, so $|E^2|=|E^1|=2|V^1|-2 = 2|V^2|-2$, which contradicts the assumption that $E^2$ is independent. 

\textit{Case 2.} Suppose on the other hand that $p \notin V^1$. Some edge $(c, p_k)$ or $(c, q_j)$ must be in $E^1$, and in $E^2$, since $c$ must have degree at least 3 in the rigidity circuit generated by $E^1$. Suppose some edge $(c,p_k)$ is in $E^1$ and $E^2$. 

We can add the edges $(p, p_k)$ and $(c,p)$ to the sets $E^1$ and $E^2$. We assumed that $p \notin V^1$, so neither edge can have originally been in $E^1$ and $E^2$. Adding the edges $(p, p_k)$ and $(c,p)$, and the point $p$, preserves the equality $|E^1|=2|V^1|-2$. Adding $p$, $(p, p_k)$, and $(c,p)$ also preserves independence of $E^2$. Furthermore, now $|V^1|=|V^2|$, so $|E^2| = |E^1| = 2|V^1|-2= 2|V^2|-2$, which contradicts that $E^2$ is independent. 

\textit{Case 3.} Suppose that $p \notin V^1$, and that some edge $(c, q_j)$ is in $E^1$ and $E^2$. 

\textit{Case 3.1.} If $e'=(p,q_j) \in E_i$, then we can proceed as in Case 2 to obtain a contradiction. 

\textit{Case 3.2.} If on the other hand $e'=(p, q_j)$ is not in $E_i$, then $E_i \cup \{e'\}$ is dependent. Else $e'$ would have been added to $E_i$, since $(c, q_j) \in E^1 \setminus \{f\} \subseteq E_i$. Since $E_i \cup \{e'\}$ is dependent, there is a rigidity circuit containing $e'$ generated by some set of edges $E^3 \subseteq E_i \cup \{e'\}$. Note that $E^4 = E^1 \setminus \{f\}$ generates a minimally rigid subgraph of $G_i$, with vertex set containing $q_i$, $q_j$ and $c$. Similarly, $E^5= E^3 \setminus \{e'\}$ generates a minimally rigid subgraph of $G_i$ with vertex set containing $q_j$ and $p$. Furthermore, $q_j \in V^4 \cap V^5$, so $E_4$ and $E_5$ generate two minimally rigid subgraphs of $G_i$ with vertex sets that intersect in at least one vertex.

\textit{Case 3.2.1.} If $V^4$ and $V^5$ intersect in more than one vertex, then we are in the situation of Lemma \ref{intersection}. In this case $E^4 \cup E^5$ generates a minimally rigid graph that does not contain $e$, but that does contain the vertices $p$ and $q_i$. Adding $e$ to this minimally rigid graph causes a dependency among the edges. As $E^4 \cup E^5 \subseteq E_i$, $E_4 \cup E_5 \cup \{e\}$ being dependent contradicts $E_i \cup \{e\}$ being independent. 

\textit{Case 3.2.2.} Suppose that $V^4$ and $V^5$ intersect in exactly one vertex. Then we are in the situation of Lemma \ref{empty_intersect}.  Note that if $(p, c)$ were already in $E^4 \cup E^5$, then $p \in V^4$, if $(p,c) \in E^4$ or $c \in V^5$, if $(p,c) \in E^5$, so one of $p$ and $c$ would belong to the intersection of $V_4$ and $V_5$. Since we already know that $q_j \in V_4 \cap V_5$, $V_4$ and $V_5$ would intersect in more than one vertex. The edge $(p,c)$ therefore cannot be in $E_4 \cup E_5$, so, by Lemma \ref{empty_intersect}, we can add the edge $(p,c)$ to $E^4 \cup E^5$ to obtain a minimally rigid graph that contains the vertices $p$ and $q_i$. Adding $e$ to this minimally rigid graph causes a dependency among the edges, and since $E^4 \cup E^5 \cup \{(p,c)\} \cup \{e\} \subseteq E_i$, this contradicts $E_i \cup \{e\}$ being independent.

We can now conclude that $e=(p, q_i)$ cannot be added to $E_i$ without causing any dependencies, so each $G_i$ is maximally independent in $G^C(S_i)$. Therefore, by induction, $G'=G_n$ is maximally independent in $G^C(S_n)=G^C(S)$, and by construction $V_n=V'=V^C$. This proves point 1 of the Lemma.

Next we want to prove that $S'$ is sharply independent, meaning that inequality $(\ref{inequality})$ holds for any subset $J \subset I'$ and its support $Q \times M$.



We will first show that if $J\subseteq I'$ with support  $Q \times M \subseteq P' \times L'$ is the set of incidences of an incidence geometry obtained from a subgraph of $G'$ in the same manner $S'$ is obtained from $G'$, then $|J| \leq |M|+2|Q|-3$.

Indeed, let $G''=(V'',E'')$ be a subgraph of $G'$ and let $I''\subseteq V''\times E''$ be the incidence relation of $G''$ when regarded as an incidence geometry. 
Let $Q= V''$, and let $M$ consist of the subsets of points that are collinear in $S^C$ and in a common star graph in $G''$. Note that a single edge is a star graph on two vertices.

A star graph representing a line $\ell$ on $k(\ell)$ points contributes $k(\ell)-1$ edges to $E''$. So for every line $\ell$ on $k(\ell) \geq 2$ points in $M$, there are $k(\ell)-2$ more edges in $E''$ than there are lines in $M$.

Similarly, a line incident to $k(\ell)$ points contributes $k(\ell)$ incidences to $J$, while the $k(\ell)-1$ edges in the star graph representing $\ell$ contribute $2k(\ell)-2$ incidences to $I''$. So for every line on $k(\ell) \geq 2$ points there are $k(\ell)-2$ more incidences in $I''$ than there are in $J$. 

Therefore, the following holds for the graph $G''=(V'',E'')$ with incidence relation $I''\subseteq V''\times E''$  and the incidence geometry $(Q,M,J)$: 
$$\begin{array}{lcl}
|Q|&=&|V''|,\\\\
|M|&=&|E''|-\Sigma_{\ell\in M } (k(\ell)-2), \mbox{ and}\\\\
|J|&=&|I''|-\Sigma_{\ell \in M} (k(\ell)-2).
\end{array}$$
Each edge in $E''$ contributes two incidences, $|I''|=2|E''|$.
Also, $|E''| \leq 2|V''|-3=2|Q|-3$ since $G'$ is independent. Therefore
$$\begin{array}{rcl}
  |J|&=&|I''|-\Sigma_{\ell\in M} (k(\ell)-2) = 2|E''|-\Sigma_{\ell \in M} (k(\ell)-2) = |M|+|E''|\\\\
  &\leq&  |M|+2|V''|-3=|M|+2|Q|-3.
\end{array}$$

We have shown that if $(Q,M,J)$ can be constructed from a subgraph $G''=(V'',E'')$ of $G'$, then $|J|\leq |M|+2|Q|-3$. If $|E''|=2|V''|-3$, then equality holds in the final inequality. In particular, the subgeometry constructed from $E'$ is $S'$. If we, as in point 3 of the Lemma, assume that $|E'|=2|V'|-3$, then $|I'| = |L'|+2|P'|-3$. 


We now consider the case that $J\subseteq I'$ with support  $Q\times M \subseteq P'\times L' $ is \textit{not} the set of incidences of an incidence geometry obtained from a subgraph of $G'$ in the same manner $S'$ is obtained from $G'$, and we will see that still $|J| \leq |M|+2|Q|-3$.

Suppose that for any incidence $(p, \ell) \in J$ between a point $p$ and a line $\ell$ in $J$, there is at least one edge $(p,q)$ in the star graph in $G'$ on vertices representing points incident to $\ell$ such that the incidence $(q, \ell) \in J$. Then the incidence structure obtained from the subgraph $G''=(V'',E'')$ with $V''=Q$ and edge set given by edges $(p,q)$ such that $(p,q) \in E'$ and the incidences $(p, \ell)$ and $(q, \ell)$ are in $J$ is $(Q,M,J)$. 

Therefore, if $J$ cannot be obtained from a subgraph of $G'$ as described above, $J$ must contain an incidence $(p,\ell)$ between  $p\in Q\subseteq P'$ and  $\ell\in M\subseteq  L'$ such that if $e=(p,q)$ is any edge incident with $p$ in the star graph in $G'$ on the vertices representing points incident to $\ell$, then the incidence $(q,\ell)\not\in J$.

\textit{Case 1.}
Suppose $p$ is the only element in $Q$ incident with $\ell$. In this case, the removal of $(p,\ell)$ from $J$ and the removal of $\ell$ from $M$ subtracts the same from both sides of the inequality, so these subtractions together preserve the inequality (\ref{inequality}). 

\textit{Case 2.}
Suppose $\ell$ is incident with some other points $p_1,\dots,p_t$ in $Q$, with $t\geq 1$.  Suppose also that $q\not\in Q$. In this case $\ell$ must be a line of the original incidence geometry $S$, with at least three incident points. Adding the missing point $q$ to $Q$ and the incidence $(q,\ell)$ to $J$ will add one to the left side and two to the right side of the inequality.

Let $c$ be the cone vertex corresponding to $\ell$.

\textit{Case 2.1.} Suppose $c\not \in Q$. Then none of the lines to $c$ in $S'$ are in $M$. We know that there are at least three such lines in $L^C$ (namely $(p,c)$, $(p_1,c)$ and $(q,c)$). By the construction of $G'$, these edges are also in $L'$.  Add $c$ to $Q$ and three of the lines incident with $c$ to $M$. This adds $5$ to the right side of the inequality  (\ref{inequality}). The three lines have a total of $6$ incidences in $S'$. Add these incidences to $J$. This adds $6$ to the left side of the inequality.   These additions (together with the addition of $q$ and $(q,\ell)$)  preserve the inequality (\ref{inequality}).

\textit{Case 2.2.}
Suppose instead that $c\in Q$.

\textit{Case 2.2.1.} Suppose that the line $\ell'$ between $c$ and $q$ is in $M$, then add the incidence between $q$ and this line. This adds one to the left side of inequality (\ref{inequality}). Together with the addition of $q$ and the incidence $(q, \ell)$, this preserves the inequality (\ref{inequality}).

\textit{Case 2.2.2.}
Suppose instead $\ell' \notin M$. Then add the line $\ell'$ to $M$ and the two incidences $(c,\ell')$ and $(p',\ell')$ to $J$. This adds one to the right side and two to the left side of the inequality (\ref{inequality}). Together with the addition of $q$ and the incidence $(q, \ell)$, this preserves the inequality (\ref{inequality}).


In Case 1 and Case 2, it is possible to add or remove incidences in such a way that the inequality (\ref{inequality}) is preserved, and so that the result is a subset of incidences that can be represented by a subgraph. The inequality (\ref{inequality}) holds for subsets of incidences that can be obtained from subgraphs of $G'$, so it must also hold for $J$.

\textit{Case 3.} Suppose $\ell$ is incident to some other points $p_1, \dots, p_t$ in $Q$ and $q \in Q$. Then we can add the incidence $(q, \ell)$ to $J$. This adds one incidence to $J$, so one to the left side of the inequality (\ref{inequality}). Therefore, if the subgeometry resulting from adding $(q, \ell)$ to $J$ satisfies the inequality (\ref{inequality}), so does the original subgeometry $(Q,M,J)$. The resulting subset of incidences can be obtained from a subgraph of $G'$, and we have seen that the inequality (\ref{inequality}) holds for such subsets. Therefore the inequality (\ref{inequality}) holds for all subsets $J \subset I'$, so $S'$ is sharply independent. 

Recall that we also saw, when considering subsets of incidences obtained from subgraphs of $G'$, that if $|E'|=2|V'|-3$, then $S'$ satisfies $|I'|=|L'|+2|P|-3$. Note that $I'$ is obtained from a subgraph of $G'$, namely $G'$ itself, so there is nothing to show in order to prove point 3 when considering subsets of incidence geometries not obtained from subgraphs of $G'$. We have therefore seen that point 3 of the Lemma holds.  

In conclusion, we have shown that $G'$ is maximally independent. We have also shown that $S'$ is sharply independent, and if we assume that $|E'| = 2|V'| - 3$, then $|I'|=|L'| + 2|P| -3$.
\end{proof}

\begin{remark}
	If we consider an arbitrary maximally independent subgraph $G'$ of a cone graph $G^C(S)$ an incidence geometry $S$, and construct $S'$ as in Lemma \ref{Construction}, $S'$ will not always be sharply independent. If we construct $S'$ from the maximally independent subgraph $G'$ constructed in Lemma \ref{Construction}, $S'$ will be sharply independent, which we will use when proving our main theorem.
\end{remark}

	

\section{Determining the infinitesimal rigidity of rod configurations}

The main theorem of this paper, Theorem \ref{Main}, relates the rigidity properties of a cone graph $G^C(S)$ in $\mathbb{R}^2$ to the rigidity properties of certain string configurations realizing the cone incidence geometry $S^C$ in $\mathbb{R}^2$.  We can therefore determine the infinitesimal rigidity of rod configurations realizing $S$ in $\mathbb{R}^2$ by determining the rigidity of any cone graph of $G^C(S)$, which can be done efficiently by the pebble game algorithm. We will discuss the computational aspects further in the next section. 

In order to state our main theorem we need a notion of genericity for linear realizations. One natural notion could be formulated in terms of line slopes - namely that the line slopes are linearly independent over $\mathbb{Q}$. However, with that definition, only incidence geometries with sharply independent sets of incidences can have proper generic linear realizations. As we are considering incidence geometries that do not necessarily have sharply independent sets of incidences, we will instead introduce the notion of regularity. We say that a linear realization $\rho$ of an incidence geometry $S=(P,L,I)$ is \textit{regular} if the rows of the concurrency matrix $M(S, \rho)$ corresponding to any sharply independent subset of $I$ are independent.

In other words, a proper linear realization is regular if any subset of incidences that can correspond to independent rows of $M(S, \rho)$ also corresponds to independent rows of $M(S, \rho)$.

\begin{theorem}
Let $S=(P,L,I)$ be an incidence geometry such that $S^C$ has a realization as a regular proper string configuration in $\mathbb{R}^2$. Then all cone graphs of $S$ are rigid in $\mathbb{R}^2$ if and only if every regular proper string configuration realizing $S^C$ in $\mathbb{R}^2$ is infinitesimally rigid.
	\label{Main}
\end{theorem}

\begin{proof}

By assumption $S^C$ has a regular proper realization as a string configuration in $\mathbb{R}^2$. We may choose the collinear trees representing the lines in the string configuration to be star graphs, as the choice of collinear trees does not affect the infinitesimal rigidity of the string configuration. Furthermore, a string configuration realizing $S^C$ in $\mathbb{R}^2$ with the collinear trees representing the lines chosen to be star graphs, is a framework of some cone graph $G^C(S)$ in non-generic position. 

Suppose first that every regular proper string configuration of $S^C$ is infinitesimally rigid. Then this framework of $G^C(S)$ in non-generic position is infinitesimally rigid. Hence if every regular proper string configuration realizing $S^C$ in $\mathbb{R}^2$ is infinitesimally rigid, then $G^C(S)$ has an infinitesimally rigid framework in $\mathbb{R}^2$, which means that $G^C(S)$ is rigid in $\mathbb{R}^2$, by Lemma \ref{GenericRigidity}. So $S$ has a cone graph which is rigid in $\mathbb{R}^2$, which by Lemma \ref{RigidConeGraph} means that all cone graphs of $S$ are rigid in $\mathbb{R}^2$.

Conversely, suppose that all cone graphs of $S$ are rigid in $\mathbb{R}^2$. Then the maximally independent subgraph constructed in Lemma \ref{Construction} is minimally rigid in $\mathbb{R}^2$, as it is a maximally independent subgraph of graph which is rigid in $\mathbb{R}^2$. So, by Lemma \ref{Construction}, the induced subgeometry $S'$ is sharply independent and satisfies $|I'| = |L'| + 2|P'| -3$. It follows that any regular proper realization of $S'$ as a string configuration in $\mathbb{R}^2$ is minimally infinitesimally rigid, by Theorem \ref{Whiteley5.2}.  Any regular proper realization of $S^C$ as a string configuration in the plane induces a regular proper realization of $S'$ as a string configuration in the plane, which must be infinitesimally rigid. This implies that the original string configuration realizing $S^C$ was also infinitesimally rigid, since the string configuration realizing $S'$ is a spanning minimally rigid subgraph of the original string configuration realizing $S^C$.
\end{proof}

Given a rod configuration of $S$ in $\mathbb{R}^2$, it is possible to find infinitely many proper string configurations realizing $S^C$. If any such string configuration is regular, that string configuration is infinitesimally rigid if and only if $S$ has a cone graph which is rigid in $\mathbb{R}^2$, by Theorem \ref{Main}. As the rod configuration realizing $S$ is infinitesimally rigid if and only if the regular proper string configuration realizing $S^C$ is infinitesimally rigid, and since it is easy to determine rigidity of graphs in the plane, this gives a way of determining the infinitesimal rigidity of rod configurations.

We also obtain the following corollary, which does not require that we know that $S^C$ has a realization as a regular proper string configuration.

\begin{corollary}
	Let $S=(P,L,I)$ be an incidence geometry. If $G^C(S)$ is flexible in $\mathbb{R}^2$, then $S^C$ does not have a realization as an infinitesimally rigid regular proper string configuration in $\mathbb{R}^2$.
\end{corollary}

\begin{proof}
	Suppose $G^C(S)$ is flexible in $\mathbb{R}^2$. If $S^C$ has a realization as a regular proper string configuration in $\mathbb{R}^2$, all its realizations as regular proper string configurations in $\mathbb{R}^2$ are infinitesimally flexible by Theorem \ref{Main}. If $S^C$ does not have a realization as a regular proper string configuration in $\mathbb{R}^2$, it also does not have a realization as an infinitesimally rigid regular proper string configuration in $\mathbb{R}^2$.
\end{proof}

Frameworks of a cone graph of an incidence geometry $S$ in the plane are body-and-joint frameworks realizing $S$, and similarly string configurations realizing $S^C$ in the plane are body-and-joint frameworks realizing $S$ such that all joints incident to a body are collinear. This is how Theorem \ref{Main} can be seen as a version of the molecular conjecture (Theorem \ref{molecular}) in the plane for general incidence geometries. The molecular conjecture and Theorem \ref{Whiteley5.2} concern incidence geometries that always have realizations as regular proper rod configurations. This is not the case for general incidence geometries, which is why we assume that there is a regular proper string configuration realizing $S^C$ in the plane in Theorem \ref{Main}.

\section{Algorithmic aspects}

Jacobs and Hendrickson introduced an algorithm for determining the independence of edges in a graph, known as the pebble game algorithm \cite{Pebbles}. The pebble game algorithm determines the rigidity of the input graph in the plane. Therefore, by Theorem \ref{Main}, we can determine the infinitesimal rigidity of regular proper rod configurations realizing an incidence geometry $S$ by applying the pebble game algorithm to a cone graph of $S$. We begin this section by introducing the algorithm itself.

\subsection{The pebble game algorithm}

The input of the pebble game algorithm is a graph $G=(V,E)$ and an ordering of the edges. The output of the algorithm is a maximally independent subgraph of $G$. From the end state of the algorithm it is also possible to see whether the maximally independent subgraph satisfies $|E|=2|V|-3$, and is minimally rigid by Theorem \ref{Geiringer-Laman}. As noted in \cite{Pebbles}, we may consider the edges in any order, although changing the order of the edges may affect which maximally independent subgraph is produced by the algorithm.

To initialize the algorithm, each vertex is given two pebbles. 

Following the ordering of the edges, the algorithm checks whether there is a total of four pebbles at the endpoints of an edge $e$. If there is, then the edge is accepted. A pebble is then removed from one of the endpoints of $e$, and $e$ is directed from that endpoint.

If there is not a total of four pebbles at the endpoints of $e$, then the algorithm searches directed paths starting at the endpoints of $e$. If a pebble is found, the direction of the edges in the path are reversed and the pebble is moved to the relevant endpoint of $e$. The edge is accepted if there is a total of four pebbles at the endpoints of $e$.

The algorithm stops when no more edges can be accepted, and the output of the algorithm is the accepted edges.






The accepted edges generate a maximally independent subgraph of $G$. The number of remaining pebbles tells us whether the graph is rigid or flexible. If there are more than three remaining pebbles, then the input graph is flexible. If all edges are accepted, then the input graph is independent, otherwise it is not. 

If there are exactly three remaining pebbles, then the input graph is rigid. If all edges are accepted, then the input graph is minimally rigid, otherwise it is rigid but not minimally rigid. 

See \cite{Pebbles} for more applications of the pebble game algorithm, and see also \cite{k_l_pebbles} for an extension of Jacobs and Hendricksons pebble game to other similar matroids.

\subsection{Determining the infinitesimal rigidity of rod configurations using the pebble game algorithm}

By Theorem \ref{Main} we can determine if a regular proper rod configuration realizing an incidence geometry $S$ is infinitesimally rigid by applying the pebble game algorithm with any ordering of the edges to any cone graph $G^C(S)$. The output of the pebble game algorithm will be some maximally independent subgraph of $G^C(S)$. By Theorem \ref{Main}, any regular proper rod configuration realizing $S$ is infinitesimally rigid if and only if the maximally independent subgraph is rigid, i.e. if there are exactly three remaining pebbles at the end state of the algorithm.

A rod configuration is \textit{minimally infinitesimally rigid} if it is infinitesimally rigid if removing any rod, but no points, results in an infinitesimally flexible rod configuration. We explore the problem of determining when a rod configuration is minimally infinitesimally rigid further in another article \cite{rods}. 

By removing a rod and using the pebble game algorithm to decide whether the resulting rod configuration is infinitesimally rigid, we can check whether the rod can be removed without the rod configuration becoming infinitesimally flexible. In this way, we can check for all rods in the rod configuration whether they can be removed, and decide whether the rod configuration is minimally infinitesimally rigid.

We note also that the pebble game can be used to construct the specific maximally independent subgraph of a cone graph $G^C(S)$ required in the proof of Lemma \ref{Construction}. Given an ordering of the lines, we can test independence of the edges in the cone of a line in the ordering specified in the proof of Lemma \ref{Construction}. We can do this for each line in the specified ordering.

Running the pebble game algorithm on the cone graph with this ordering of the edges will give the maximally independent subgraph of $G^C(S)$ constructed in Lemma \ref{Construction}. 

A special feature of this minimally rigid subgraph $G'$ is that frameworks of $G'$ where vertices corresponding to points incident to some common line are collinear are also infinitesimally rigid. This is not in general true of all minimally rigid subgraphs of some rigid cone graph $G^C(S)$. See Figure \ref{Not_inf_rig} for an example.

\begin{figure}
	\begin{center}
	\begin{tabular}{cc}
	\begin{tikzpicture}
\filldraw[black] (0,0) circle (2pt);
\filldraw[black] (-1.5,-3) circle (2pt);
\filldraw[black] (1.5,-3) circle (2pt);
\filldraw[black] (0.3,-1.1) circle (2pt);
\filldraw[black] (-0.25,-1.5) circle (2pt);
\filldraw[black] (0.3,-2.75) circle (2pt);
\filldraw[black] (0.2,-1.75) circle (2pt);

\filldraw[black] (1.5,-0.8) circle (2pt);
\filldraw[black] (-1.4,-1) circle (2pt);
\filldraw[black] (-0.4,-3.5) circle (2pt);
\filldraw[black] (0.5,-2) circle (2pt);

\draw[thick] (0,0) -- (0.25, -1.1);
\draw[thick] (0.25,-1.1)--(1.5,-3);

\draw[thick] (0,0) -- (-0.25,-1.5);
\draw[thick] (-0.25,-1.5)--(-1.5,-3);

\draw[thick] (0.3,-1.1) -- (0.2,-1.75);
\draw[thick] (-1.5,-3) -- (0.2,-1.75);

\draw[thick] (1.5,-0.8)--(0,0);
\draw[thick] (1.5,-0.8)--(0.3,-1.1);
\draw[thick] (1.5,-0.8)--(1.5,-3);

\draw[thick] (-1.4,-1)--(0,0);
\draw[thick] (-1.4,-1)--(-0.25,-1.5);
\draw[thick] (-1.4,-1)--(-1.5,-3);

\draw[thick] (0.5, -2) -- (-1.5,-3);
\draw[thick] (0.5, -2) -- (0.3,-1.1);
\draw[thick] (0.5, -2) -- (0.2,-1.75);

\draw[thick] (1.5,-3) -- (0.3,-2.75);
\draw[thick] (-1.4,-3) -- (0.3,-2.75);

\draw[thick] (-0.4,-3.5) -- (1.5,-3);
\draw[thick] (-0.4,-3.5) -- (-1.5,-3);
\end{tikzpicture}
&
	\begin{tikzpicture}
\filldraw[black] (0,0) circle (2pt);
\filldraw[black] (-1.5,-3) circle (2pt);
\filldraw[black] (1.5,-3) circle (2pt);
\filldraw[black] (0.5,-1) circle (2pt);
\filldraw[black] (-0.5,-1) circle (2pt);
\filldraw[black] (0,-3) circle (2pt);
\filldraw[black] (0,-1.5) circle (2pt);

\filldraw[black] (1.5,-1) circle (2pt);
\filldraw[black] (-1.5,-1) circle (2pt);
\filldraw[black] (0,-3.5) circle (2pt);
\filldraw[black] (0.5,-2) circle (2pt);

\draw[thick] (0,0) -- (1.5,-3) -- (-1.5,-3)--(0,0);
\draw[thick] (-1.5,-3) -- (0.5,-1);

\draw[thick] (1.5,-1)--(0,0);
\draw[thick] (1.5,-1)--(0.5,-1);
\draw[thick] (1.5,-1)--(1.5,-3);

\draw[thick] (-1.5,-1)--(0,0);
\draw[thick] (-1.5,-1)--(-0.5,-1);
\draw[thick] (-1.5,-1)--(-1.5,-3);

\draw[thick] (0.5, -2) -- (-1.5,-3);
\draw[thick] (0.5, -2) -- (0.5,-1);
\draw[thick] (0.5, -2) -- (0,-1.5);

\draw[thick, ->] (0,-3) -- (0,-3.3);
\draw[thick] (0,-3.5) -- (1.5,-3);
\draw[thick] (0,-3.5) -- (-1.5,-3);
\end{tikzpicture}
	\end{tabular}
	\caption{A minimally rigid subgraph of the cone graph $G^C(S)$ from Figure \ref{Cone_graph_rods} that does not remain infinitesimally rigid when the relevant points are placed on a line.}	
		\label{Not_inf_rig}
	\end{center}
\end{figure}
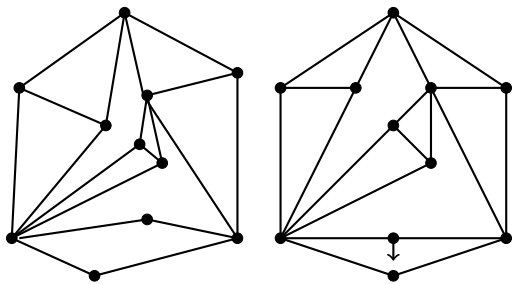

\subsection{Complexity}

Determining whether an edge is independent of the previously accepted edges requires $\mathcal{O}(v)$-time, where $v$ is the number of vertices in the input graph (Corollary 2.9, \cite{Pebbles}). Hence the algorithm runs in $\mathcal{O}(ve)$ time, where $e$ is the number of edges in the input graph. If the input graph is a cone graph $G^C(S)$ of an incidence geometry $S=(P,L,I)$, then $v= |P| + |L|$ is the number of vertices in $G^C(S)$, and $e=2|I|-|L|$ is the number of edges in $G^C(S)$. Hence determining the rigidity of a regular proper rod configuration can be done in $\mathcal{O}((|P|+|L|)(2|I|-|L|))$-time.

The chosen ordering of the edges does not affect the complexity, so finding the maximally independent subgraph constructed in Lemma \ref{Construction} can also be done in $\mathcal{O}((|P|+|L|)(2|I|-|L|))$-time.

In order to determine the minimal rigidity of a rod configuration, we may need to run the pebble game algorithm $|L|$ times, so determining the minimal rigidity of a rod configuration can be done in $\mathcal{O}(|L|(|P|+|L|)(2|I|-|L|))$-time.

\section{Concluding remarks}

In this article, we reduced the problem of determining the infinitesimal rigidity of any regular, proper rod configuration in the plane to determining the rigidity of its cone graph in the plane. We prove the equivalence between infinitesimal rigidity of regular proper rod configurations and rigidity of cone graphs in Theorem \ref{Main}, which can be seen as an extension of Theorem \ref{Whiteley5.2} to incidence geometries that do not necessarily have realizations as independent body-and-joint frameworks. 

It is interesting to ask whether Theorem \ref{Main} can be formulated and proved in arbitrary dimension, and worth noting that the molecular conjecture does hold in any dimension. There are a few key details of our proof that cannot be immediately generalized to higher dimensions, specifically those that rely on Theorem \ref{Geiringer-Laman} and Theorem \ref{Whiteley5.2}. 

The algorithm to determine whether a rod configuration is minimally infinitesimally rigid is based on the algorithm to determine whether a rod configuration is infinitesimally rigid, by in turn removing each rod and checking whether the resulting rod configuration is infinitesimally rigid. Because the algorithm to determine whether a rod configuration is infinitesimally rigid is polynomial time, so is the algorithm for determining whether a rod configuration is minimally infinitesimally rigid, but understanding minimal infinitesimal rigidity of rod configurations better could lead to even better algorithms.


\section{Acknowledgements}

The work has been supported by the Knut and Alice Wallenberg Foundation Grant 2020.0001 and 2020.0007.

\end{document}